\documentclass[11pt,a4paper,twoside]{article}
\usepackage{bm, amsmath, amssymb, amsthm}
\usepackage{graphicx}

\def\<{\langle}
\def\>{\rangle}
\def\eps{\varepsilon}

\def\RR{\mathbb{R}}
\def\vbeta{{\bm b}}

\newcommand\tr{\operatorname{Tr}}

\newcommand\id{ \operatorname{id}}

\def\vol{\operatorname{vol}}

\def\eq{\hspace*{-2.0mm}&=&\hspace*{-2.0mm}}
\def\plus{\hspace*{-2.0mm}&+&\hspace*{-2.0mm}}
\def\minus{\hspace*{-2.0mm}&-&\hspace*{-2.0mm}}

\newtheorem{corollary}{Corollary}

\newtheorem{definition}{Definition}
\newtheorem{example}{Example}
\newtheorem{remark}{Remark}
\newtheorem{lem}{Lemma}
\newtheorem{proposition}{Proposition}
\newtheorem{thm}{Theorem}

\author{Vladimir Rovenski\footnote{Department of Mathematics, University of Haifa,
        \ e-mail: {\tt vrovenski@univ.haifa.ac.il}}
        \ and \
        Pawe\l \  Walczak\footnote{Katedra Geometrii,
        Uniwersytet \L\'{o}dzki,
\ e-mail: {\tt pawel.walczak@wmii.uni.lodz.pl}}
}

\title{Deforming convex bodies in Minkowski geometry}

\begin{document}

\date{}

\maketitle

\begin{abstract}
We introduce and study deformation $T_{\vbeta,\phi}$ of Minkowski norms in $\RR^n$, determined by a set
$\vbeta=(\beta_1,\ldots,\beta_p)$ of linearly independent 1-forms and a smooth positive function $\phi$ of $p$ variables.
In~particular, the $T_{\vbeta,\phi}$-image of a Euclidean norm $\alpha$ is a Minkowski norm,
whose indicatrix is a rotation hypersurface with a $p$-dimensional axis passing through the origin.
For~$p=1$, our deformation genera\-lizes construction of $(\alpha,\beta)$-norm;
the last ones form a rich class of ``computable" Minkowski norms and play an important role in Finsler geometry.
We~use compositions of $T_{\vbeta,\phi}$-deformations with $\vbeta$'s of length $p$ to define
an equiva\-lence relation $\overset{p}\sim$ on the set of all Minkowski norms in $\RR^n$.
We apply M. Matsumoto result to characterize the cases when the Cartan torsions of a norm and its $T_{\vbeta,\phi}$-image
either coincide or differ by a C-reducible~term.


\end{abstract}

\section*{Introduction}

\smallskip

In the paper we introduce and study deformations $T_{\vbeta,\phi}$ of Minkowski norms in $\RR^n$, determined by a sequence $\vbeta=(\beta_1,\ldots,\beta_p)$ of linearly independent 1-forms $\beta_i$ and a positive function $\phi$ of $p$ variables.
In particular, $T_{\vbeta,\phi}(\alpha)$ (the images of Euclidean norm $\alpha$) are $(\alpha,\vbeta)$-norms \cite{r-fin3},
whose indicatrix is a rotation hypersurface with a $p$-dimensional axis passing through the origin.
The $(\alpha,\beta)$-norms have also recently been generalized in \cite{js1} as homogeneous combinations of several Minkowski
norms and one-forms.
For~$p=1$, our deformations $T_{\beta,\phi}$ genera\-lize construction \cite{shsh2016} of $(\alpha,\beta)$-norms,
which form a rich class of ``computable" Minkowski norms and play an important role in differential geometry.
We expect that our $T_{\vbeta,\phi}$-deformations will also find many applications in Minkowski Geometry \cite{thom} as well as in Finsler Geometry \cite{shsh2016}.
Our question is about mapping between two pairs $(B_i, q_i)\ (i=1,2)$, $B_i$ being a convex body and $q_i$ -- a point of the interior of~$B_i$.
\textit{When $(B_1, q_1)$ can be $T_{\vbeta,\phi}$-mapped by one-step or a sequence of deformations onto $(B_2, q_2)$}?

More exactly, we pose the following
\textbf{Problem}. Consider the space $Cpt(\RR^n)\  (n\ge2)$ of compact pointed subsets of $\RR^n$ equipped with the Hausdorff distance $d_H$. Given two compact convex bodies $B$ (e.g. a ball) and $B'$ and points ${\cal O}\in B$ and ${\cal O}'\in B'$ in $\RR^n$, can one approximate (in $Cpt(\RR^n)$) $B'$ by the unit ball $\bar B=\{\bar F=1\}$ of a Minkowski norm $\bar F$ having origin $\bar{\cal O}$ close to ${\cal O}'$ and being equivalent (in the sense of Definition~\ref{Def-rel-p} with either $p =1$ or arbitrary $p$) to the norm $F$ of origin ${\cal O}$ and $B=\{ F = 1\}$. That is, given $\epsilon>0$, can one find $\bar F$ as above for which $d_H(B',\bar B)<\epsilon$ and $\max\{F({\cal O}'-\bar{\cal O}), F(\bar{\cal O}-{\cal O}')\}<\epsilon$. If not, can one do something like that with the Gromov-Hausdorff distance $d_{GH}$ (see \cite{gr}) replacing~$d_H$?

In the paper, we define
(using compositions of $T_{\vbeta,\phi}$-deformations with $\vbeta$'s of length $\le p$)
and study an equiva\-lence relation $\overset{p}\sim$ on the set of all Minkowski norms in $\RR^n$.
We show that any axisymmetric convex body $B$ in $\RR^2$ can be moved to a unit disc by $T_{\beta,\phi}$-deformation,
but {we have no example of Minkowski norm in $\RR^2$ nonequivalent $\overset{1}\sim$ to Euclidean norm}.
We~prove that any ellipsoid in $\RR^3$ can be deformed to a sphere in a finite number of steps with $T_{\beta,\phi}$-deformations,
but {we cannot say the same for a general convex body in $\RR^3$ with a plane of symmetry}.
Answers to the above questions will require a thorough study of Cartan torsion.
In the paper, we characterize the cases when the Cartan torsions of the Minkowski norm $F$ and its image
$\bar F=T_{\vbeta,\phi}(F)$ coincide or their difference is a C-reducible~term.

\section{Construction}

Recall that a \textit{Minkowski norm} on a vector space $\RR^{n}\ (n>1)$ is a function $F:\RR^n\to[0,\infty)$ with the  properties of regularity, positive 1-homogeneity and strong convexity, see~\cite{shsh2016}:

M$_1:$ $F\in C^\infty(\RR^n\setminus \{0\})$,

M$_2:$ $F(\lambda\,y)=\lambda F(y)$ for $\lambda>0$ and $y\in \RR^n$,

M$_3:$ For any $y\in \RR^n\setminus \{0\}$, the following symmetric bilinear form is {positive definite}:
\begin{equation}\label{E-acsiom-M3}
 g_y(u,v)=\frac12\,\frac{\partial^2}{\partial s\,\partial t}\,\big[F^2(y+su+tv)\big]_{|\,s=t=0}\,,\quad u,v\in\RR^n.
\end{equation}

\noindent

By M$_2$--\,M$_3$, $g_{\lambda y}=g_{y}\ (\lambda>0)$ and $g_y(y,y)=F^2(y)$.
As a result of M$_3$, the indicatrix $S:=\{y\in \RR^n: F(y)=1\}$ is a closed,
convex smooth hypersurface that surrounds the origin.

\begin{remark}\label{R-1.1new}\rm
In some interesting cases (when $F$ is not positive definite or smoothly inextendible in some directions,
so that these directions must be removed from the domain of $F$)
Minkowski norms are defined only in conic domains $A\subset\RR^n$, e.g., \cite{js1}.
Such cases called \textit{pseudo-Minkowski} or \textit{conic Minkowski norms}, are considered below only in examples
(e.g., slope or Kropina norms).
\end{remark}

The following symmetric trilinear form is called the \textit{Cartan~torsion} for $F$:
\begin{equation}\label{E-Ctorsion}
 C_y(u,v,w)=\frac14\,\frac{\partial^3}{\partial r\,\partial s\,\partial t}\,\big[F^2(y+ru+sv+tw)\big]_{|\,r=s=t=0}\,,\quad u,v,w\in\RR^n,
 \end{equation}
where
$y\ne0$.
Note that $C_y(u,v,y)=0 $ and $C_{\lambda y}=\lambda^{-1}C_y$ for $\lambda>0$\,.
A 1-form
 $I_y(u)= \tr_{g_y} C_y(u,\,\cdot\,,\cdot)$,
is called the \textit{mean Cartan torsion}, its vanishing characterizes Euclidean norm among all Minkowski norms, e.g. \cite{shsh2016}.
In coordinates, we have
\[
 C_{ijk}=\frac14\,[F^2]_{y^iy^jy^k} = \frac12\,\frac{\partial g_{ij}}{\partial y^k},\quad
 g_{ij} = \frac12\,[F^2]_{y^iy^j},\quad
 C_{k} = C_{ijk} g^{ij},
\]
where $C_k$ are components of the mean Cartan torsion.
The \textit{angular metric tensor of $F$},
\[
 K_y(u,v)=g_y(u,v)-g_y(y,u)\,g_y(y,v)/F^2(y),
\]
in coordinates has the view $K_{ij}=F\cdot F_{y^i y^j} = g_{ij} - g_{ip}\,y^p g_{iq}\,y^q/F^2$.

A Minkowski norm $F$ is called \textit{semi-C-reducible}, if its Cartan torsion has the form \cite{shsh2016}
\begin{equation}\label{E-semi-C-r}
 C_{ijk} = \frac{p}{n+1}
 \big(K_{ij}C_k+K_{jk}C_i+K_{ki}C_j\big) +\eps\frac{1-p}{C^2}\,C_iC_jC_k,
\end{equation}
where $p\in\RR$,  and
\begin{equation*}
 \eps(C)^2=g^{ij}C_iC_j\ne0.
\end{equation*}
$F$ in $\RR^n\ (n\ge3)$ is called \textit{C-reducible}, if its Cartan torsion
has the form of \eqref{E-semi-C-r} with $p=1$,
\[
 C_{ijk} = \frac{1}{n+1} \big(K_{ij}C_k+K_{jk}C_i+K_{ki}C_j\big) .
\]
By \cite[Proposition~5]{ma92}, any $(\alpha,\beta)$-norm,
 $F = \alpha\,\phi(\beta/\alpha)$,
 with nonzero mean Cartan torsion is semi-C-reducible.
In dimension greater than two,
any C-reducible Minkowski
norm has the Randers ($\phi=1+s$) or the Kropina ($\phi=1/s$) type, see \cite[Theorem 2.2 (M.~Matsumoto)]{shsh2016} and Remark~\ref{R-RanKrop}.

\subsection{General case}

Here, we define a deformation of Minkowski norms by using a sequence of $p$ linearly independent 1-forms
in $\RR^n$ and a positive function of $p$ variables.

\begin{definition}\label{D-01}\rm
Let $F$ be a Minkowski norm on $\RR^n$,
$\phi:\prod_{\,i=1}^{\,p}[-\delta_{i},\delta_{i}]\to(0,\infty)$ a smooth positive function
for some $p\le n$ and $\vbeta=(\beta_1,\ldots,\beta_p)$ a sequence of
linearly independent 1-forms on $\RR^n$ of the norm $F(\beta_i)<\delta_i$.
Then, the $T_{\vbeta,\phi}$-\textit{deformation} of $F$
(or, of convex body defined by $\{F\le1\}$) is the following mapping:
\begin{equation}\label{E-ab-def}
 T_{\vbeta,\phi}: F \mapsto F\cdot\phi(s),\quad s=(s_1,\ldots,s_p),\quad s_i=\beta_i/F.
\end{equation}
\end{definition}

The image $T_{\vbeta,\phi}(\alpha)$ of a Euclidean norm $\alpha$ is called $(\alpha,\vbeta)$-\textit{norm}, see \cite{r-fin3},
its indicatrix is a rotation hypersurface in $\RR^n$
with the $p$-dimensional axis span$\{\beta_1^\sharp,\ldots,\beta_p^\sharp\}$.
For $p=1$ (and $\beta_1=\beta$),~\eqref{E-ab-def} defines the $T_{\beta,\phi}$-transformation, see Section~\ref{sec:p-1a} and Example~\ref{Ex-01} below.

\begin{remark}\rm
Our norm \eqref{E-ab-def} can be viewed as a canonical form of so called $(F,\vbeta)$-norm, defined by
 $\bar F = \sqrt{L(F,\beta_{1},\ldots,\beta_{p})}$,
where $L:\RR^{p+1}\to\RR$ is a continuous function with the following properties:
a)~$L$ is smooth and positive away from 0,
b)~$L$ is positively homogeneous of degree 2, i.e., $L(\lambda\,y) = \lambda^2\,L(y)$ for all $\lambda>0$.
Indeed, set $\phi(s_1,\ldots,s_p)=\sqrt{L(1,s_{1},\ldots,s_{p})}$, then $\bar F = F\,\phi(\beta_1/F,\ldots,\beta_p/F)$.
For $p=1$, such combinations were named $\beta$-changes in \cite{sh1984} and $(F,\beta)$-norms in \cite[p.~845]{js1}.
\end{remark}

 By direct calculation we~get

\begin{proposition}\label{prop:compos}
The composition
$T_{\vbeta,\phi_2}\circ\,T_{\vbeta,\phi_1}$ with the same $\vbeta$
has the form $T_{\vbeta,\phi}$, where
\[
 \phi=\phi_1(s)\,\phi_2(s/\phi_1(s)).
\]
\end{proposition}

Define real functions
of variables $(s_1,\ldots,s_p)$:
\begin{equation}\label{E-n-rhos}
 \rho = \phi\big(\phi-\sum\nolimits_{\,i}s_i\dot\phi_i\big),\quad
 \rho_{0}^{ij} = \phi\,\ddot\phi_{ij}+\dot\phi_i\dot\phi_j,\quad
 \rho_{1}^{i} = \phi\,\dot\phi_i -\sum\nolimits_{\,j} s_j\big(\phi\,\ddot\phi_{ij}+\dot\phi_i\,\dot\phi_j\big),
\end{equation}
where $\dot\phi_i=\frac{\partial\phi}{\partial s_i}$,
$\ddot\phi_{ij}=\frac{\partial^2\phi}{\partial s_i\partial s_j}$, etc.
Note that the following relations hold:
\[
 \dot\rho_i = \rho_1^i,\quad
 \ddot\rho_{ij} = (\rho_1^i)'_j = -s_k(\rho_0^{ik})'_j,
\]
and if $s_i=0\ (1\le i\le p)$ then $\rho=1$, $\rho_{1}^{i}=0$ and $\rho_{0}^{ij}=0$.

\begin{example}\label{Ex-01}\rm
If $F$ is the Euclidean norm $\alpha(y)={\<y,y\>}^{1/2}$
then
$T_{\beta,\phi}(\alpha)$ is
an $(\alpha,\beta)$-\textit{norm}, see~\cite{shsh2016}.
Some progress was achieved for particular cases of $(\alpha,\beta)$-norms, e.g.
Randers norm $\alpha+\beta$, introduced by a physicist G.\,Randers to study the unified field theory;
Kropina norm $\alpha^{2}/\beta$,
first introduced by L.\,Berwald in connection with a Finsler plane with rectilinear extremal, and investigated by V.K.\,Kropina;
slope norm $\frac{\alpha^2}{\alpha-\beta}$,
introduced by M.~Matsumoto to study the time it takes to negotiate any given path on a hillside.
These can be viewed as images of $\alpha$ under $T_{\beta,\phi}$-deformations (for $p=1$).
We~define similarly particular $T_{\beta,\phi}$-deformations of Minkowski norms in~$\RR^n$.

(i)~The \textit{Randers deformation} appears for $\bar F=F+\beta$ with $F(\beta)<1$, i.e., $\phi(s)=1+s$.
We~have $\rho=1+s,\ \rho_0=\rho_1=1$.

(ii)~\textit{Generalized Kropina deformations} appear for $\bar F = F^{l+1}/\beta^l\ (l>0)$, i.e., $\phi(s) = 1/s^l\ (s>0)$.
For $l=1$ we get the~\textit{Kropina deformation}.
Then $\rho=2/s^2$, $\rho_0=3/s^4$,~$\rho_1=-4/s^3$.

(iii)~The \textit{slope-deformation} appears for $\bar F = \frac{F^2}{F-\beta}$ with $F(\beta)<1$, i.e., $\phi(s) = \frac1{1-s}$.
We~have $\rho=\frac{1-2s}{(1-s)^3},\ \rho_0=\frac{3}{(1-s)^4}$ and $\rho_1=\frac{1-4s}{(1-s)^4}$.

(iv) The~\textit{quad\-ratic deformation} appears for $\bar F=(F+\beta)^2/F$ with $F(\beta)<1$, i.e., $\phi(s)=(1+s)^2$.
We have $\rho=(1-s)(1+s)^3,\ \rho_0=6(1+s)^2$ and $\rho_1= 2(1-2s)(1+s)^2$.

Note that Kropina and slope norms are not Minkowski norms: their $F$'s are not defined on the whole $\RR^n\setminus\{0\}$
(have singularities) and their $g_y$'s are not positive definite.
Figure~\ref{F-deform} shows quadratic deformation
of indicatrix of $m$-root norm $F=((y^1)^m+(y^2)^m)^{1/m}$ in $\RR^2$ for $m=2,\ldots,8$ and two cases: a)~$0.3\,dy^2$, and b)~$0.3 (dy^1+dy^2)$ of $\beta$.
\end{example}

\begin{figure}[ht]
\begin{center}
\includegraphics[scale=.5,angle=0]{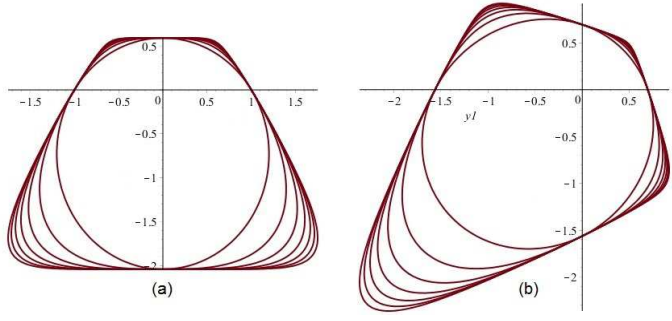}
\end{center}
\caption{Indicatrix of $m$-root norm in $\RR^2$ after quadratic deformation.\label{F-deform}}
\end{figure}

Assume in the paper that $\rho>0$, i.e.,
\begin{equation}\label{E-p-cond1}
 \phi-\sum\nolimits_{\,i}s_i\dot\phi_i>0,
\end{equation}
otherwise metric in \eqref{E-c-value0} is not positive
for small $s_i$.
Set $\tilde p_y=\rho_1^i p_{yi}$, where $p_{yi}= \beta_i^{\sharp y}-s_i y/F(y)$.

\begin{thm}
Let $\bar F=T_{\vbeta,\phi}(F)$, then the bilinear forms, see \eqref{E-acsiom-M3}, are related as
\begin{eqnarray}\label{E-c-value0}
\nonumber
  \bar g_y(u,v) \eq \rho\,g_y(u,v) +\rho_{0}^{ij}\beta_i(u)\beta_j(v) \\
 \plus\rho_{1}^{i}(\beta_i(u) g_y(y,v) +\beta_i(v) g_y(y,u))/F(y) -\rho_{1}^{i}\beta_i(y) g_y(y,u)\,g_y(y,v)/F^{3}(y).
\end{eqnarray}
The Cartan torsions of $\bar F$ and $F$, see \eqref{E-Ctorsion}, are related as
\begin{eqnarray}\label{E-Ctorsion2}
\nonumber
 && 2\,\bar C_y(u,v,w) = 2\,\rho\,C_y(u,v,w) \\
\nonumber
 && +\,\big( K_y(u,v) g_y(\tilde p_{y},w) +K_y(v,w) g_y(\tilde p_{y},u) +K_y(w,u) g_y(\tilde p_{y},v)\big)/F(y) \\
 && +\sum\nolimits_{\,i,j,k}(\dot\phi_i\ddot\phi_{jk}+\dot\phi_j\ddot\phi_{ik}
 +\dot\phi_k\ddot\phi_{ij}
 +\phi\,\dddot\phi_{ijk})\,g_y(p_{yi},u)\,g_y(p_{yj},v)\,g_y(p_{yk},w)/F(y).
\end{eqnarray}
\end{thm}

\proof
From \eqref{E-acsiom-M3} and \eqref{E-ab-def} we find
\begin{eqnarray}\label{En-g-K}
\nonumber
 &&\hskip-4mm \bar g_y(u,v) = {[\bar F^2/2]_F} K_y(u,v)/F(y) +{[\bar F^2/2]_{FF}}\,g_y(y,u)\,g_y(y,v)/F^{2}(y) \\
 &&\hskip-5mm +\sum\nolimits_{\,i}({[\bar F^2/2]_{F\beta_i}}/{F(y)})\big(g_y(y,u)\beta_i(v) +g_y(y,v)\beta_i(u)\big)
 {+}\sum\nolimits_{\,i,j}[\bar F^2/2]_{\beta_i\beta_j}\,\beta_i(u)\beta_j(v).
\end{eqnarray}
Calculating derivatives of
$\frac12\,\bar F^2=\frac12\,F^2\phi^2(\beta_1/F,\ldots,\beta_p/F)$,
\begin{eqnarray}\label{En-F2}
\nonumber
 &&
 [\bar F^2/2]_F={F}\rho,\quad
 [\bar F^2/2]_{\beta_i}=F\phi\,\dot\phi_i,\quad
 [\bar F^2/2]_{F\beta_i}=\rho_{1}^{i},\quad
 [\bar F^2/2]_{\beta_i\beta_j}=\rho_{0}^{ij},\\
 && [\bar F^2/2]_{FF} = \rho +(\sum\nolimits_{\,i}s_i\,\dot\phi_i)^2 +\phi\sum\nolimits_{\,i,j}s_i\,s_j\,\ddot\phi_{ij}
\end{eqnarray}
and comparing \eqref{E-c-value0} and \eqref{En-g-K}, completes the proof of \eqref{E-c-value0}.
Recall that if $H(z^1,\ldots, z^q)$ is a positively homogeneous of degree $r$ function then $H_{z^k}\,z^k=r H$.
 The 0-homogeneity of $[\bar F^2/2]_{\mu\,\nu}$ in variables $F,\{\beta_k\}$ yields
\[
 [\bar F^2/2]_{F\mu\,\nu}\,F + \sum\nolimits_{\,k}\,[\bar F^2/2]_{\beta_k\mu\,\nu}\,\beta_k=0
\]
for $\mu,\nu\in\{F,\beta_1,\ldots,\beta_p\}$; hence,
\begin{eqnarray*}
 && [\bar F^2/2]_{F\beta_i\beta_j} = -\sum\nolimits_{\,k}(\beta_k/F)\,[\bar F^2/2]_{\beta_i\beta_j\beta_k}, \\
 && [\bar F^2/2]_{FF\beta_i} = \sum\nolimits_{\,j,k}(\beta_j/F)(\beta_k/F)\,[\bar F^2/2]_{\beta_i\beta_j\beta_k}, \\
 && [\bar F^2/2]_{FFF} = -\sum\nolimits_{\,i,j,k}(\beta_i/F)(\beta_j/F)(\beta_k/F)\,[\bar F^2/2]_{\beta_i\beta_j\beta_k}.
\end{eqnarray*}
Using this, we calculate the Cartan torsion \eqref{E-Ctorsion} after $T_{\vbeta,\phi}\,$-deformation as
\begin{eqnarray}\label{E-Ctorsion2a}
\nonumber
 && 2\,\bar C_y(u,v,w) = 2\,\rho\,C_y(u,v,w) \\
\nonumber
 \plus \sum\nolimits_{\,i}[\bar F^2/2]_{F\beta_i}
 \big( K_y(u,v) g_y(p_{yi},w)+K_y(v,w) g_y(p_{yi},u) + K_y(w,u) g_y(p_{yi},v)\big)/F(y) \\
 \plus  \sum\nolimits_{\,i,j,k}[\bar F^2/2]_{\beta_i\beta_j\beta_k}\, g_y(p_{yi},u)\,g_y(p_{yj},v)\,g_y(p_{yk},w).
\end{eqnarray}
Next, using equalities \eqref{En-F2}, we obtain
\begin{equation*}
 [\bar F^2/2]_{\beta_i\beta_j\beta_k} = \big(\dot\phi_i\,\ddot\phi_{jk}+\dot\phi_j\,\ddot\phi_{ik}
 +\dot\phi_k\,\ddot\phi_{ij} +\phi\,\dddot\phi_{ijk}\big)/F(y).
\end{equation*}
The above, and comparing \eqref{E-Ctorsion2a} and \eqref{E-Ctorsion2} completes the proof of \eqref{E-Ctorsion2}.
\qed

\begin{remark}\rm
By \eqref{E-c-value0}, $\bar g_y$ of $T_{\vbeta,\phi}\,$-deformation $\bar F$ can be viewed as a perturbed metric $g_y$ of $F$.
The~second line term on the RHS of \eqref{E-Ctorsion2} is the symmetric product of the 2-tensor $K_y/F(y)$ and the 1-form $\iota_{\,\tilde{p}_y} g_y$ while the third line term is a linear combination of symmetric products of 1-forms $\iota_{\,p_{yi}}g_y\ (i=1,\ldots,p)$.  Therefore, the difference $\bar C_y - \rho\,C_y$ is semi-C-reducible when $p = 1$ (see also \cite{shsh2016}), but in general is not when  $p > 1$.
\end{remark}

\begin{proposition}\label{P-03}
Let $\phi$ be a smooth positive function of class C$^{\,2}$ defined in a neighborhood of the origin $O$ of $\RR^p$.
Then there exists $\delta>0$ such that \eqref{E-p-cond1} holds and
for any Minkowski norm $F$ and arbitrary $1$-forms $\beta_1, \ldots, \beta_p$ of $F$-norm less than $\delta$;
then \eqref{E-ab-def} determines the  Minkowski norm~$\bar F$.
\end{proposition}

\begin{proof}
The formula (\ref{E-c-value0}) shows that the inner products $\bar g_y$ for $y\in F^{-1}(1)$
depend uniformly on the 1-forms $\beta_i\ (i = 1,\ldots, p)$. For
$\beta_1 = \ldots = \beta _p = 0$, $\bar g_y = \rho\,g_y$ is positive definite. Compactness of the $F$-unit sphere implies the statement.
\end{proof}

\begin{remark}\label{R-2.1}\rm
One can get a more accurate than Proposition~\ref{P-03} sufficient conditions for positive definiteness of $\bar g_y$.
By \cite[Theorem~4.1]{js1}, $\bar g_y>0$ for $\bar F=\sqrt{L}=F\phi(\frac{\beta_1}F, \ldots,\frac{\beta_k}F)$, see \eqref{E-ab-def}, if
(for $L=F^2\phi^2$)

\smallskip
(a)	$L_{,F} > 0$, \quad (b) ${\rm Hess}(L) > 0$ (positive definite matrix).

\smallskip\noindent
For~$p=1$, we have $L_{,F} > 0$ and ${\rm Hess}(L)>0$ if and only if, see \cite[Eqn. (4.13)]{js1},
\[
 \phi - s\,\dot\phi>0,\quad
 \ddot\phi>0.
\]
For $p>1$, we find $L_{,F} = 2\,F\phi(\phi-\sum_{\,i} s_i\dot\phi_i)$, thus, the inequality $ L_{,F}>0$
is guaranteed by assumption \eqref{E-p-cond1}.
Calculation of ${\rm Hess}(L)$, see condition (b), is similar to the case of $p=1$ in the proof of \cite[Proposition~4.10]{js1}.
Put $L=F^2\psi(s_1,\ldots,s_p)$, where $s_i=b_i/F$, $\psi=\phi^2$ and $p>1$. Then $L_{,b_i} = F\,\dot\psi_{i}$ and
\begin{eqnarray*}
 L_{,FF} \eq 2\,\psi - 2\sum\nolimits_{\,i} s_i\,\dot\psi_i +\sum\nolimits_{\,i,j} s_i s_j\,\ddot\psi_{ij},\\
 L_{,F b_i} \eq  \dot\psi_i  - \sum\nolimits_{j} s_j\,\ddot\psi_{ij},\quad
 L_{,b_i b_j} = \ddot\psi_{ij}.
\end{eqnarray*}
Applying elementary transformations of matrices to the first row and first column of ${\rm Hess}(L)$, we find
\[
 {\rm Hess}(L) \sim \Psi =
 \left(
   \begin{array}{cccc}
     2\,\psi & \dot\psi_1 & \cdots & \dot\psi_p \\
     \dot\psi_1 & \ddot\psi_{11} & \cdots & \ddot\psi_{1p} \\
     \cdots & \cdots & \cdots & \cdots \\
     \dot\psi_p & \ddot\psi_{1p} & \cdots & \ddot\psi_{pp} \\
   \end{array}
 \right),\quad
 \det{\rm Hess}(L)=\det\Psi.
\]
In our case, ${\rm Hess}(L)>0$ is satisfied when $\det\Psi>0$ and $\psi$ is a convex function.
Note that conditions in Proposition~\ref{P-03} do not require $\psi=\phi^2$ to be convex.
\end{remark}

 We restrict ourselves to regular $T_{\vbeta,\phi}\, $-deformations alone, i.e., $\det\bar g_y\ne0\ (y\ne0)$.
By \eqref{E-c-value0},
\begin{equation}\label{E-c-value0-b}
 \bar g_y(u,v) = \rho\,g_y(u,v) +(\rho_{0}^{ij}+\eps^{-1}\rho_{1}^{i}\rho_{1}^{j})\beta_i(u)\beta_j(v)
 -\eps\,g_y(\tilde Y,u)\,g_y(\tilde Y,v),
\end{equation}
where
\begin{equation}\label{E-Y-eps-value}
 \tilde Y=\eps^{-1}\rho_1^i\beta_i^{\sharp,y} - y/F(y),\quad \eps=s_j\rho_1^j.
\end{equation}
The volume forms of metrics $\bar g_y$ and $g_y$ for $y\ne0$~are
\[
 \,{\rm d}\vol_{g_y}(e_1,\ldots,e_{n})=\sqrt{\det g_y(e_i,e_j)},\quad
 {\rm d}\vol_{\bar g_y}(e_1,\ldots,e_{n})=\sqrt{\det\bar g_y(e_i,e_j)},
\]
where $\{e_1,\ldots,e_{n}\}$ is a basis of $\RR^n$.
Then
 ${\rm d}\vol_{\bar g_y}=\sigma_{y}\,{\rm d}\vol_{g_y}$
for some function $\sigma_{y}>0$.
Define vectors $\tilde\beta_k=q_{k}^i\beta_i\ (k\le p)$, where
$q_k=(q_{k}^1,\ldots,q_{k}^p)\in\RR^p$ are unit eigenvectors with eigenvalues $\lambda^k$
of the
matrix $\{\rho_{0}^{ij}+\eps^{-1}\rho_{1}^{i}\rho_{1}^{j}\}$.
 Then \eqref{E-c-value0-b} takes the form, which can be used to find
 $\sigma_{y}$:
\begin{equation*}
 \bar g_y(u,v) = \rho\,g_y(u,v) + \sum\nolimits_{\,i}\lambda^i\,\tilde\beta_i(u)\,\tilde\beta_i(v)
 -\eps\,g_y(\tilde Y,u)\,g_y(\tilde Y,v) .
\end{equation*}

\begin{example}
\rm
For $p=2$, using
$\tilde b_{ij} =g_y(\tilde\beta_i,\tilde\beta_j)$,
$\tilde\beta_i=q_i^1\beta_1+q_i^2\beta_2$ and $\eps=\rho_1^1 s_1+\rho_1^2 s_2$,
we get
\begin{eqnarray*}
\nonumber
 \hskip-4mm \sigma_{y} \eq \rho^{n-1}\big( \rho^2
 +\rho(\lambda^1\tilde b_{11}+\lambda^2\tilde b_{22}) - \rho\,\eps g_y(\tilde Y,\tilde Y)
 +\lambda^1\lambda^2(\tilde b_{11} \tilde b_{22}-\tilde b_{12}^2) \\
\nonumber
 \minus \eps g_y(\tilde Y,\tilde Y)(\lambda^1\tilde b_{11}+\lambda^2 \tilde b_{22})
 +\lambda^1\eps g_y(\tilde\beta_1,\tilde Y)+\lambda^2\eps g_y(\tilde\beta_2,\tilde Y)
 +\lambda^1\lambda^2\,\eps/\rho\big[ \tilde b_{11} g_y(\tilde\beta_2,\tilde Y)^2 \\
 \plus \tilde b_{22} g_y(\tilde\beta_1,\tilde Y)^2
 +\tilde b_{12} g_y(\tilde Y,\tilde Y)^2 -\tilde b_{11}\tilde b_{22} g_y(\tilde Y,\tilde Y)
 -2 \tilde b_{12} g_y(\tilde\beta_1,\tilde Y)\, g_y(\tilde\beta_2,\tilde Y)\big] \big).
\end{eqnarray*}
A special case of $T_{\vbeta,\phi}\,$-deformation for $p=2$
is a shifted $T_{\phi,\beta_1}$-deformation $\bar F=F\phi(\beta_1/F)+\beta_2$.

(a)~For a shifted Kropina deformation $\bar F=F^2/{\beta_1}+\beta_2$, we get
\begin{eqnarray*}
 && \rho={(2+s_1)(1+s_1+s_1s_2)}/{s_1^2},\quad
 \rho_{1}^{1}=-{(4+3s_1+2s_1s_2)}/{s_1^3},\quad \rho_{1}^{2}={(2+s_1)}/{s_1},\\
 && \rho_{0}^{11}={(3+2s_1+2s_1s_2)}/{s_1^4},\qquad \rho_{0}^{12}=\rho_{0}^{21}=-1/{s_1^2},\qquad
  \rho_{0}^{22}=1.
\end{eqnarray*}

(b)~For a shifted slope deformation $\bar F=F^2/({F-\beta_1})+\beta_2$,
we~have
\begin{eqnarray*}
 && \rho=\frac{(1-2s_1)(1+s_2-s_1s_2)}{(1-s_1)^3},\quad
 \rho_{1}^{1}=\frac{1+2s_1(s_1s_2-s_2-2)}{(1-s_1)^4},\quad \rho_{1}^{2}=\frac{1-2s_1}{(1-s_1)^2},\\
 && \rho_{0}^{11}={(3-2s_1s_2+2s_2)}/{(1-s_1)^{4}},\quad \rho_{0}^{12}=\rho_{0}^{21}=1/{(1-s_1)^{2}},\quad
  \rho_{0}^{22}=1.
\end{eqnarray*}
Figure~\ref{F-3d-deform} shows the image of
$(y^1)^4+(y^2)^4+(y^3)^4=1$ in~$\RR^3$ for
a) quadratic deformation with $p=1$ and $\beta=0.3\,dy^3$,
b) shifted quadratic deforma\-tion with $p=2$ and $\beta_1=0.3\,dy^2,\ \beta_2=0.3\,dy^3$.
\begin{figure}[ht]
\begin{center}
\includegraphics[scale=.6,angle=0]{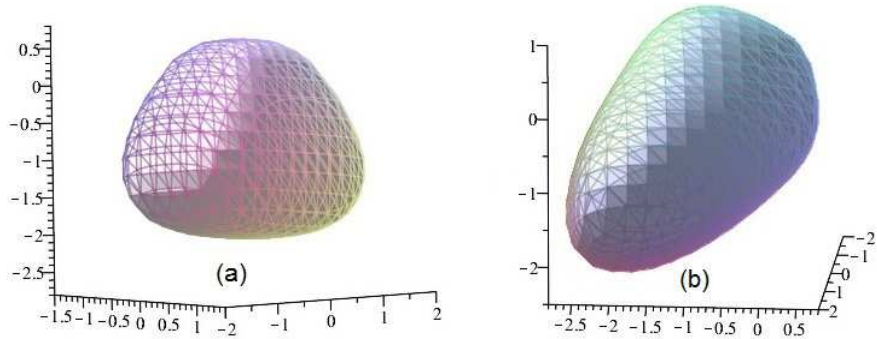}
\end{center}
\caption{Indicatrix of $m$-root norm in $\RR^3$ after shifted quadratic deformation.\label{F-3d-deform}}
\end{figure}
\end{example}

\begin{definition}
\rm
Let $G$ be a subgroup of $GL(n,\RR)$. Then a Minkowski norm $F$ on $\RR^{n}$ is called $G$-\textit{invariant} if the following
holds for some affine coordinates $(y^1,\ldots, y^{n})$ of~$\RR^n$:
\begin{equation*}
  F(y^1,\ldots, y^{n}) = F(f(y^1,\ldots, y^{n})),\quad y\in \RR^n,\ f\in G.
\end{equation*}
\end{definition}

\begin{proposition}[see \cite{r-fin3}]\label{P-symmetry}
A Minkowski norm $F$ on $\RR^{n}$ is $G$-invariant, where
\[
 G=\Big\{\Big(\,\begin{matrix}
     C & {0} \\ {0} & \id_p \\
   \end{matrix}\,\Big),\ \
   C\in O(n-p,\RR)\,\Big\},
\]
if and only if there exist linear independent 1-forms $\beta_i\ (1\le i\le p)$, for which $F$ is $(\alpha,\vbeta)$-norm.
\end{proposition}

Using Proposition~\ref{P-symmetry}, we can prove the following.

\begin{corollary}
Let $F$ and $\bar F=T_{\vbeta,\phi}(F)$ be Minkowski norms in $\RR^n$, where
$\vbeta=(\beta_1,\ldots,\beta_{p})$ and $\phi$ is a function of $p<n$ variables.
Suppose that $F$ is an $(\alpha,\beta)$-norm
with 1-form $\beta=\beta_{p+1}$
transverse to ${\rm span}\{\beta_{1},\ldots,\beta_{p}\}$.
Then $\bar F=T_{\vbeta,\psi}(\alpha)$,
where $\vbeta=(\beta_1,\ldots,\beta_{p+1})$, $\psi$ is a function of $p+1$ variables and $\alpha$ is a Euclidean~norm.
\end{corollary}


\subsection{Case of $p=1$}
\label{sec:p-1a}

Here, we illustrate and clarify some results about our transformations of Minkowski norms for the case $p=1$.
By Proposition~\ref{P-symmetry} for $p=1$ and $n=2$, any axisymmetric convex body $B$ in $\RR^2$ (for example, indicatrix of an $m$-root norm)
can be moved to a unit disc by $T_{\beta,\phi}$-deformation.

 The next lemma is used to compute the volume forms,
it extends the Silvester's determinant identity
 $\det(\id_n+C_1 P_1^t) = 1+C_1^t P_1$,
where $C_1$ and $P_1$ are $n$-vectors (columns).

\begin{lem}\label{L-det-F}
Given real $c_1,c_2$, vectors $b^1,b^2$ in $\RR^n$, and reversible symmetric $n\times n$ matrix $a_{ij}$, define
matrices $A_{ij}=a_{ij}+c_1b_i^1b_j^1$ and $g_{ij}=a_{ij}+c_1b_i^1b_j^1+c_2b_i^2b_j^2$. Then
\begin{eqnarray*}
 \det[A_{ij}] \eq \det[a_{ij}](1+c_1|b^1|^2_a),\\
 \det[g_{ij}] \eq \det[a_{ij}]\cdot[(1+c_1|b^1|^2_a)(1+c_2|b^2|^2_a) -c_1c_2\<b^1,b^2\>_a^2].
\end{eqnarray*}
\end{lem}

\proof The first claim is straightforward, \cite[Lemma~4.1]{shsh2016}.
If $1+c_1|b^1|^2_a\ne0$ then the inverse matrix
$A^{kl}=a^{kl}-\frac{c_1}{1+c_1|b^1|^2_a}b_k^1b_l^1$ exists.
For any vectors $u,v$ we get
\[
 \<u,v\>_A = A^{kl}u_k v_l = \big(a^{kl} -\frac{c_1b^{1k}b^{1l}}{1+c_1|b^1|^2_a}\big) u_k v_l
 = \<u,v\>_a - \frac{c_1}{1+c_1|b^1|^2_a}\<b^1,u\>_a\<b^1,v\>_a.
\]
Hence,
\[
 |b^2|^2_A = A^{kl}b_k^2b_l^2 = |b^2|^2_a - \frac{c_1}{1+c_1|b^1|^2_a}\<b^1,b^2\>^2_a.
\]
Using the first claim, we get $\det[g_{ij}]{=}\det[A_{ij}](1{+}c_2|b^2|^2_A)$.
The above yields the second claim.\hfill$\square$

\smallskip

We will specify Proposition~\ref{P-03} for $p=1$ (and generalize \cite[Lemma~2]{shsh2016}, that is for $F=\alpha$).

\begin{proposition}[see Corollary~4.25 in\cite{js1}]\label{P-2.5}
Let $F = F_0/(\phi/F_0)$, choose $y_0\in A\setminus 0$ (see Remark~\ref{R-1.1new} and \cite[Proposition 4.10]{js1}),
and put $s_0 = \beta(y_0)/F_0(y_0)$. In the case of dimension $n > 2$, the fundamental
tensor $g_{y_0}$ is positive definite if and only if
\begin{subequations}
\begin{eqnarray}\label{E-Yava1}
 && \phi(s_0)-s_0\,\dot\phi(s_0) > 0, \\
\label{E-Yava2}
 && \phi(s_0)-s_0\,\dot\phi(s_0) +(\|\beta\|^2_{g^0_{y_0}}-s_0^2)\,\ddot\phi(s_0) > 0,
\end{eqnarray}
\end{subequations}
and, in the case of $n = 2$, $g_{y_0}$ is positive definite if and only if \eqref{E-Yava2} holds.
\end{proposition}
The following result is a consequence of Proposition~\ref{P-2.5}.
We prove it here just for the convenience of a reader.

\begin{corollary}
The function $\bar F=F\phi(s)$, where $s=\beta/F$, is a Minkowski norm
for any Minkowski norm $F$
and a 1-form $\beta$ on $\RR^n$ with $F(\beta)< b_0$,
if and only if the function $\phi:(-b_0, b_0)\to\RR$
satisfies
\begin{equation}\label{E-phi-cond}
 \phi(s)>0,\quad
 \phi(s)-s\,\dot\phi(s) +(b^2-s^2)\,\ddot\phi(s) > 0 ,
\end{equation}
where $s$ and $b$ are arbitrary real numbers with $|s| \le b < b_0$.
\end{corollary}

\proof
Assume that \eqref{E-phi-cond} is satisfied.
Taking $s\to b$ in (\ref{E-phi-cond}), we see that
\begin{equation}\label{E-phi-cond2}
 \phi-s\,\dot\phi>0
\end{equation}
for any $s$ with $|s| < b_0$.
Consider the following families of functions and metrics:
\begin{equation*}
 \phi_t(s) :=1 - t + t\phi(s),\quad
 \bar F_t=F\phi_t(\beta/F),\quad
 \bar g_y^{\,t}:\ \bar g_{ij}^{\,t}=(1/2)[\bar F^2_t]_{y^i\,y^j}.
\end{equation*}
Note that for any $t\in[0,1]$ and any $s, b$ with $|s| \le b < b_0$,
\begin{eqnarray*}
 && \phi_t - s\dot\phi_t = 1 - t + t(\phi - s\dot\phi) > 0,\\
 &&
 \phi_t - s\dot\phi_t + (b^2 - s^2)\,\ddot\phi_t = 1 - t + t\{\phi - s\dot\phi + (b^2 - s^2)\,\ddot\phi\} > 0.
\end{eqnarray*}
By Lemma~\ref{L-det-F}, for $\bar F_t = F\phi_t(\beta/F)$ we find a formula (as in the case of $F=\alpha$, see \cite{shsh2016})
\begin{equation*}
  \det \bar g^{\,t}_y = \phi_t^{n+1}(\phi_t-s\,\dot\phi_t)^{n-2}[\phi_t-s\,\dot\phi_t +(b^2-s^2)\,\ddot\phi_t]\det g_y.
\end{equation*}
Thus, $\det\bar g_y^{\,t} > 0$ for all $t\in[0,1]$, see also \eqref{E-F001e} in what follows.
Since $\bar g^{\,0}_{y}=g_y$ is positive definite, then $\bar g^{\,t}_y$ is positive definite for any $t\in[0,1]$.
Hence, $\bar F_t$ is a Minkowski norm for any $t\in[0,1]$.
In particular, $\bar F=\bar F_1$ is a Minkowski norm.

Conversely, assume that $\bar F = F\phi(\beta/F)$ is a Minkowski norm for any
Minkowski norm $F$ and 1-form $\beta$ with $b := F(\beta) < b_0$.
Then $\phi(s)>0$ for any $s$ with $|s| < b_0$.
If $n$ is even, then $\det\bar g_y > 0$ implies that \eqref{E-phi-cond} holds for any $s$ with $|s| \le b$.
If $n>1$ is odd, then $\det\bar g_y > 0$ implies that the inequality $\phi(s)-s\dot\phi(s)\ne0$ holds for any $s$ with $|s| \le b$.

Since $\phi(0)>0$, the above inequality implies that \eqref{E-phi-cond2} holds
for any $s$ with $|s|<b$. Since $b$ with $0\le b\le b_0$ is arbitrary,
we conclude that \eqref{E-phi-cond2} holds for any $s$ with $|s|<b_0$.
Finally, we can see that $\det\bar g_y > 0$ implies that \eqref{E-phi-cond} holds for any $s$ and $b$ with $|s| \le b < b_0$.
\hfill$\square$

\begin{definition}[see Definition~\ref{D-01} for $p=1$]
\rm
Given a smooth positive function $\phi:(-b_0, b_0)\to\RR$ satisfying \eqref{E-phi-cond} and a 1-form $\beta$ on $\RR^n$,
the $T_{\beta,\phi}$-\textit{deformation} of a Minkowski norm $F$ on $\RR^n$ is the Minkowski norm
 $\bar F=F\phi(\beta/F)$.
Functions \eqref{E-n-rhos} for $p=1$ become functions of one variable $s=\beta/F$,
defined by the same formulas as for $(\alpha,\beta)$-norm in \cite{shsh2016}:
\begin{equation*}
 \rho = \phi(\phi-s\,\dot\phi),\quad
 \rho_{0} = \phi\,\ddot\phi+\dot\phi^2,\quad
 \rho_{1} = \phi\,\dot\phi -s(\phi\,\ddot\phi+\dot\phi^2).
\end{equation*}
\end{definition}

In a similar way, one can define $T_{\beta,\phi}$-\textit{deformation}
of a convex body in $\RR^n$ given by $\{F\le1\}$.
Put
 $p_{y}=\beta^{\sharp y}-s y/F(y)$,
where $\beta^{\sharp y}$ is defined by equality $g_y(\beta^{\sharp y},u)=\beta(u)$.

Formulas \eqref{E-c-value0} and \eqref{E-Ctorsion2}
for $p=1$, i.e., $\bar F=F\phi(\beta/F)$, generalize result on $(\alpha,\beta)$-norm in \cite{shsh2016}:
\begin{eqnarray}\label{E-c-value0-1}
\nonumber
 \bar g_y(u,v) \eq \rho\,g_y(u,v) +\rho_0\beta(u)\beta(v) +\rho_1(\beta(u) g_y(y,v) +\beta(v) g_y(y,u))/F(y) \\
 \minus \rho_1\beta(y) g_y(y,u)\,g_y(y,v)/F^{3}(y),\\
\label{E-c-value0-2}
\nonumber
 2\,\bar C_y(u,v,w) \eq 2\,\rho\,C_y(u,v,w)
 + (3\dot\phi\,\ddot\phi+\phi\,\dddot\phi) g_y(p_{y},u)\,g_y(p_{y},v)\,g_y(p_{y},w)/F(y)\\
 \plus\rho_1\big( K_y(u,v) g_y(p_y,w) +K_y(v,w) g_y(p_y,u) +K_y(w,u) g_y(p_y,v)\big)/F(y) .
\end{eqnarray}
By~\eqref{E-c-value0-2},
$T_{\beta,\phi}$-deformation changes Cartan torsion
adding a semi-C-reducible component.

Set~$\tilde Y = s^{-1}p_{y} = s^{-1}\beta^{\sharp y}-y/F(y)$ and $\eps=s\rho_1$, see \eqref{E-Y-eps-value}.
Then \eqref{E-c-value0-1} takes the equivalent form
\begin{equation}\label{E-c-value4}
 \bar g_y(u,v) = \rho\,g_y(u,v) + (\rho_0+\rho_1^2/\eps)\,\beta(u)\,\beta(v) -\eps\,g_y(\tilde Y,u)\,g_y(\tilde Y,v).
\end{equation}
From \eqref{E-c-value4} and Lemma~\ref{L-det-F} we get
the relation for the volume form ${\rm d}\vol_{\bar g_y}=\sigma_{y}\,{\rm d}\vol_{g_y}$:
\begin{eqnarray}\label{E-F001e}
\nonumber
 \sigma_{y} \eq \rho^{n-2}( \rho_0\rho_1 s^3+\rho_1^2 s^2 +(\rho-\rho_0b^2)\rho_1 s +(\rho\rho_0-\rho_1^2)b^2 +\rho^2)\\
 \eq \phi^{n+1}(\phi-s\,\dot\phi)^{n-2}[\,\phi-s\,\dot\phi+(b^2-s^2)\,\ddot\phi\,].
\end{eqnarray}

\begin{example}\rm
Let the indicatrix of Minkowski norm $\bar F$ in $(\RR^n,\alpha)$ be a unit sphere shifted by vector $d_1 e_1$ with $|d_1|<1$.
Then $\bar F$ has $(\alpha,\beta)$-type.
Indeed, assuming
$\bar F=\alpha\,\phi(\beta/\alpha)$, we get
\[
 \big(\sum\nolimits_{i=1}^n y_i^2\big)^{1/2}\,\phi\big(d\,y_1/\big(\sum\nolimits_{i=1}^n y_i^2\big)^{1/2}\,\big)
 =\big((y_1-d)^2+\sum\nolimits_{i=2}^n y_i^2\big)^{1/2},
\]
where $\beta(y)=d\,y_1$.
Put $s=d\,y_1/(\sum_{i} y_i^2)^{1/2}$. Assuming $(y_1-d)^2+\sum_{i=2}^n y_i^2=1$, we get $d\,y_1=s(s+(s^2+1-d^2)^{1/2})$.
Then we find $\phi(s)=1/(s+(s^2+1-d^2)^{1/2})$.
Similar result $\bar F=F\,\phi(\beta/F)$ we get for $F=(y_1^2+\psi^2(y_2,\ldots,y_n))^{1/2}$ with arbitrary function $\psi$,
see Fig.~\ref{F-shift} for $n=2$ and $\psi(y_2)=y_2^4$.
\begin{figure}[ht]
\begin{center}
\includegraphics[scale=.55,angle=0]{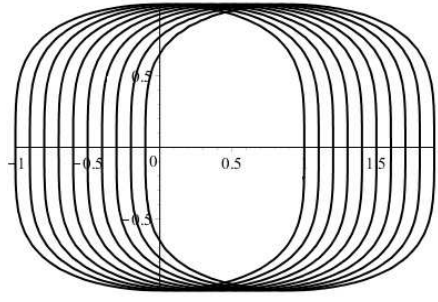}
\end{center}
\caption{Indicatrix of $\bar F$ in $\RR^2$ by $d e_1$-shift for $10d=0,\ldots,9$.\label{F-shift}}
\end{figure}
\end{example}


\begin{proposition}\label{C-3.3}
Let a Minkowski norm $F$ in $\RR^n$ can be deformed to the Euclidean norm $\alpha$
in $p<n$ steps of $T_{\beta,\phi}$-deformations with $p$ linearly independent 1-forms. Then indicatrix of $F$ has a $p$-dimensional axis of rotation.
\end{proposition}

\proof
By conditions and Theorem~\ref{prop:inverse-o}, Euclidean norm $\alpha$ can be deformed to $F$ in $p<n-1$ steps
(of  $T_{\beta,\phi}$-deformations).
After the first step, the indicatrix has 1-dimensional axis of rotation,
and each step increases the dimension of axis by one.
\hfill$\Box$

\section{The equivalence relation for Minkowski norms}

We use compositions of $T_{\vbeta,\phi}$-deformations (e.g. for $p=1$) to define and study an ``equiva\-lence relation"
on the set of all Minkowski norms in $\RR^n$.

\subsection{General case}

The following theorem shows that our deformations are invertible with the same structure.

\begin{thm}\label{prop:inverse-o}
For any $T_{\vbeta,\phi}$-deformation of $F$ to $\bar F$ (Minkowski norms in $\RR^n$) satisfying \eqref{E-p-cond1} there exists inverse $T_{\vbeta,\psi}$-deformation of $\bar F$ to $F$ (with the same $\vbeta$) satisfying $\psi - \sum_it_i\,\dot\psi_i>0$.
In particular, any $(\alpha,\vbeta)$-norm can be $T_{\vbeta,\phi}$-deformed in one step to Euclidean norm $\alpha$.
\end{thm}

\proof
Let $\bar F=F\phi(\vbeta/F)$ be our deformation, where $\phi:\Pi\to(0,\infty)$ obeys \eqref{E-p-cond1}.
We~are looking for inverse $T_{\vbeta,\psi}$-deformation (whose image of $\bar F$ is $F$,
i.e., $F=\bar F\psi(\vbeta/\bar F)$). Thus,
 $\psi\big({s}/{\phi({s})}\,\big) = 1/{\phi({s})}$,
where ${s}=\vbeta/F$. The mapping $\Phi({s}) = {s}/\phi({s})$ moves points along the rays through the origin.
To show that $\Phi(\lambda\,{s})$ is monotone in $\lambda$, we calculate the derivative
\[
 \frac d{d\lambda}\,\Phi((1+\lambda){s}/\phi((1+\lambda){s}))|_{\,\lambda=0}
 =({\phi-\sum\nolimits_{\,i} s_i\,\dot\phi_i})/{\phi^2} \overset{\eqref{E-p-cond1}} > 0.
\]
Hence, there exists mapping $\Phi^{-1}$ and the function
 $\psi({t}) = 1/\phi\circ\Phi^{-1}({t})$
is uniquely defined on a certain domain.
 The mapping $\Psi({t}) = {t}/\psi({t})$ moves points along the rays through the origin.
 Observe that $\Psi(\lambda\,{t})$ is monotone in $\lambda$,
 and has positive derivative,
\[
 0 < \frac d{d\lambda}\,\Psi((1+\lambda){t}/\psi((1+\lambda){t}))|_{\,\lambda=0} =({\psi-\sum\nolimits_{\,i} t_i\,\dot\psi_i})/{\psi^2}.
\]
Hence, condition \eqref{E-p-cond1} is satisfied for $\psi$ of $t=\vbeta/\bar F$.
\hfill$\square$

\begin{definition}\label{Def-rel-p}
\rm
We write $\bar F\overset{p}\sim F$ for Minkowski norms in $\RR^n$ and $p\le n$ if there are $T_{\vbeta^i,\phi_i}$-deformations $(i\le m)$ with $\vbeta^i$ of length $\le p$ such that $F_1=F\phi_1(\frac{\vbeta^1}{F}), \ldots, \bar F=F_{m-1}\phi_m(\frac{\vbeta^m}{F_{m-1}})$.~Set
\[
 [F]_p=\{\bar F\in{\rm Mink}^n: \bar F\overset{p}\sim F\}.
\]
\end{definition}

\begin{proposition}
The  relation $\bar F\overset{p}\sim F$ is an equivalence relation on the set of Minkowski norms.
\end{proposition}

\proof
By
Theorem~\ref{prop:inverse-o}, the relation $\overset{p}\sim$ is reflexive, symmetric and transitive.
\hfill$\Box$


\begin{remark}\rm
One may apply the equivalence relation $\overset{p}\sim$ of Minkowski norms to Finsler metrics.
For two Finsler metrics on $M$, we write $F\overset{p}\sim F'$, if $F_x\overset{p}\sim F'_x$ for all $x\in M$.
\end{remark}

\begin{proposition}\label{Prop-equivalence2}
 The class $[\alpha]_p$ on $(\RR^n,\alpha)$ is invariant under ``rotations" (by orthogonal matrices) and homotheties.
\end{proposition}

\proof
For $\phi\equiv\lambda\in\RR_+$ we obtain $\bar F=\lambda\,F$; thus, $\lambda\,F\overset{p}\sim F$ for any positive $\lambda$.
If $\bar F$ is the image of $F$ under $T_{\vbeta,\phi}$-deformation, then
$\lambda\,\bar F$ is the image of $\lambda\,F$ under $T_{\lambda\,\vbeta,\phi}$-deformation.

Let $A\in O(n)$ be an orthogonal matrix. If
\[
 F_1=\alpha\,\phi_1((\beta_1,\ldots,\beta_p)/\alpha),\quad
 \bar F_1=\alpha\,\phi_1((\bar\beta_1,\ldots,\bar\beta_p)/\alpha),
\]
where $\bar\beta_i = \beta_i\circ A$, then $\bar F_1(y)= F_1(A(y))$ for all $y$. Similarly, for a set of transformations.
\hfill$\Box$

\begin{thm}\label{T-03}
Let $F$ and $\bar F=T_{\vbeta,\phi}(F)$ be Minkowski norms in $\RR^n\ (n\ge3)$,
where $\phi\ne1$ is a smooth positive function of $p<n$ variables,
and $\vbeta=(\beta_1,\ldots,\beta_p)$ is a set of linearly independent 1-forms.
If $\tilde F = |\bar F^2- F^2|^{1/2}$ is a C-reducible Minkowski norm,
then $F=T_{\vbeta,\phi_1}(\tilde F)$ and $\bar F=T_{\vbeta,\phi_2}(\tilde F)$
for some $\phi_1(s)$ and $\phi_2(s)=\phi_1(s)\phi(s/{\phi_1(s)})$, where $\tilde F$
is a Randers norm in $\RR^n$.
\end{thm}

\begin{proof}
Let $\phi<1$ (the case $\phi>1$ is similar). Then $\tilde F^2:=F^2-\bar F^2$ is a positive (on $\RR^n\setminus\{0\}$)
2-homogeneous function and the following equality is satisfied:
\begin{equation}\label{E-C-reduce}
 \tilde F^2 = F^2(1-\phi^2(\vbeta/F)) .
\end{equation}
Thus, $F^2\le (1/m) \tilde F^2$, where $m=\min\{1-\phi^2(\vbeta(y)):\,y\in S_F\}>0$, i.e., the indicatrix of $\tilde F$
is surrounded by the image of $S_F$ after $m$-homothety. Since $S_F$ is compact, $S_{\tilde F}$ is compact as well.
The proof of \cite[Theorem 2.2 (M.~Matsumoto)]{shsh2016} requires the inverse matrix for $g_y$ in \eqref{E-acsiom-M3},
which exists, for example, when $g_y$ is positive definite.
By~our assumptions, $\tilde F$ is a Minkowski norm with a C-reducible Cartan torsion.
Following \cite[Theorem~2.2]{shsh2016} and Remark~\ref{R-RanKrop} below, we find  that $\tilde F$ is a Randers norm.
(Here, we can exclude the case of a Kropina norm because its indicatrix is non-compact).
By~\eqref{E-C-reduce}, $\tilde F=T_{\vbeta,\psi}(F)$ with $\psi=(1-\phi^2)^{-1/2}$.
By \eqref{E-cond2} (which is always satisfied when $\phi\ne1$ and $s=\vbeta/F$ is ``small"), we find that
$\psi(s)$ satisfies \eqref{E-p-cond1}. By Theorem~\ref{prop:inverse-o}, there exists an inverse defor\-mation
$F=T_{\vbeta,\phi_1}(\tilde F)$ for some $\phi_1(s)$.
The~formula for $\bar F$ follows from Lemma~\ref{prop:compos}.
\end{proof}

\begin{remark}\label{R-tildeL}\rm
Given Minkowski norms $F$ and $\bar F=T_{\vbeta,\phi}(F)$ in $\RR^n\ (n\ge3)$,
one can find sufficient conditions for a 1-homogeneous function $\tilde F=|\bar F^2-F^2|^{1/2}$ to be a Minkowski norm.
Let $\phi<1$ (the case $\phi>1$ is similar). Set $\widetilde L=F^2\big(1-\phi^2(\beta_1/F, \ldots,\beta_k/F)\big)$.
 The~sufficient conditions for $\tilde g_y>0$ (of $\tilde F$) are, see Remark~\ref{R-2.1},

\smallskip
(a)	$\widetilde L_{,F} > 0$, \quad (b) ${\rm Hess}(\widetilde L) > 0$ (positive definite matrix).

\smallskip\noindent
We find $\widetilde L_{,F} = 2 F\big(1-\phi(\phi-\sum_{\,i} s_i\dot\phi_i)\big)$. Thus, the inequality $\widetilde L_{,F}>0$ is
reduced to
\[
 \phi(\phi-\sum\nolimits_{\,i} s_i\dot\phi_i) < 1,
\]
and is guaranteed by the following assumption for $\phi\ne1$:
\begin{equation}\label{E-cond2}
 (\phi-1)\big(\phi(\phi-\sum\nolimits_i s_i\,\dot\phi_i) -1\big)>0.
\end{equation}
Put $\psi=\phi^2$. Applying elementary transformations of matrices to ${\rm Hess}(\widetilde L)$,
where $\widetilde L=F^2-L$ and $L=F^2\phi^2$, and using Remark~\ref{R-2.1}, we find
\[
 {\rm Hess}(\widetilde L) \sim \widetilde\Psi =
 \left(
   \begin{array}{cccc}
     2\,(1-\psi) & -\dot\psi_1 & \cdots & -\dot\psi_p \\
     -\dot\psi_1 & -\ddot\psi_{11} & \cdots & \ddot\psi_{1p} \\
     \cdots & \cdots & \cdots & \cdots \\
     -\dot\psi_p & -\ddot\psi_{1p} & \cdots & -\ddot\psi_{pp} \\
   \end{array}
 \right),\quad
 \det{\rm Hess}(\widetilde L)=\det\widetilde\Psi.
\]
In our case, ${\rm Hess}(\widetilde L)>0$ is satisfied when $\det\widetilde\Psi>0$ and $-\psi$ is a convex function.
\end{remark}

\begin{remark}\label{R-RanKrop}\rm
Take a (single) Minkowski $(\alpha, \beta)$-norm $F$ with a C-reducible Cartan torsion~$C$. Extend it to the (unique) Minkowskian structure $\tilde F$  on the whole $\RR^n$. At each point $x\in\RR^n$, Cartan torsion $\tilde C(x)$ equals $C$
(up to the canonical isomorphism between $\RR^n$ and $T_x \RR^n$).
Therefore, $\tilde C$  is C-reducible and comes from a Randers or Kropina $\tilde F$.
Namely,
\[
 a(x)\tilde F^2+2(\beta_i(x)\,y^i) \tilde F +a_{ij}(x) y^i y^j = 0,
\]
and $\tilde F$ is Randers when $a(x)\ne0$, while $\tilde F$ is Kropina for $a(x)=0$.
In particular, $F$ itself is of Randers or Kropina type.
Thus, result \cite[Theorem~2.2 (M.~Matsumoto)]{shsh2016} is ``pointwise".
\end{remark}

\begin{corollary}
Let $F$ and $\bar F=T_{\vbeta,\phi}(F)$ be Minkowski norms in $\RR^n\ (n\ge2)$, where $\phi$ is a smooth positive function of $p$ variables and $\vbeta=(\beta_1,\ldots,\beta_p)$ is a set of linear independent 1-forms.
If~Cartan torsions of $\bar F$ and $F$ coincide and \eqref{E-cond2} is satisfied then
$\phi(\vbeta/F)\ne1$, $F=T_{\vbeta,\phi_1}(\alpha)$ and $\bar F = T_{\vbeta,\phi_2}(\alpha)$ with the same given $\vbeta$
and some $\phi_1$ and $\phi_2(s)=\phi_1(s)\phi(\frac s{\phi_1(s)})$, where $\alpha$ is a Euclidean norm.
\end{corollary}

\proof
By conditions, $\tilde F^2:=F^2-\bar F^2$ is a 2-homogeneous function with vanishing Cartan torsion.
By definition of Cartan torsion, $\tilde F^2$ is a quadratic form on $\RR^n$.
As in the proof of Theorem~\ref{T-03}, we show that indicatrix of $\tilde F$ is compact.
Hence, the quadratic form $\tilde F^2$ is either positive or negative definite:
$\tilde F^2=\pm\alpha^2$, where $\alpha$ is a Euclidean norm in $\RR^n$.
 If $\tilde F=\alpha$ then $\phi(\vbeta/F)<1$ and
\begin{equation}\label{E-C-0}
 1-\phi^2(\vbeta/F) = \alpha^2/F^2,
\end{equation}
and similarly for $\tilde F=-\alpha$.
By \eqref{E-C-0}, $\alpha = T_{\vbeta,\psi}(F)$ with $\psi=(1-\phi^2)^{-1/2}$.
By \eqref{E-cond2}, function $\psi(s)$ satisfies \eqref{E-p-cond1}, see the proof of Theorem~\ref{T-03}.
By Theorem~\ref{prop:inverse-o}, there exists inverse transformation
$F=T_{\vbeta,\phi_1}(\alpha)$ for some $\phi_1(s)$.
The~formula for $\bar F$ follows from Proposition~\ref{prop:compos}.
\hfill$\Box$

\subsection{Case of $p=1$}

The next proposition shows that  any two Euclidean norms in $\RR^n$ are equivalent.
One can take Cartesian coordinates for the first Euclidean norm such that $\alpha^2(y)=\sum\nolimits_{i=1}^n (y_i)^2$
and the indicatrix of the
second norm is an ellipsoid given by
$\sum\nolimits_{i=1}^n d_i^2(y_i)^2=1$.

\begin{proposition}
Any two Euclidean norms, $\bar\alpha$ and $\alpha$,
in $\RR^n$
are $\overset{1}\sim$ equivalent.
Moreover,
we have $\bar\alpha = T_{\vbeta,\phi}(\alpha)$
using one transformation with $\vbeta$ of length $n$.
\end{proposition}

\proof
By~Proposition~\ref{Prop-equivalence2} (with homotheties), we can assume
$\alpha(y)=\sum\nolimits_{i=1}^n (y_i)^2$ and $\bar\alpha(y)=\sum\nolimits_{i=1}^n d_i^2(y_i)^2$ with all $d_i\in(0,1)$.
Taking $\phi(s)=\sqrt{1 - s^2}$ and $\beta=(1 - d_1^2)^{1/2}dy_1$,
we transform the unit sphere in $(\RR^n,\alpha)$ into the ellipsoid of axes $(1/d_1, 1,\ldots, 1)$.
Then, take $\beta_1 = (1- d_2^2)^{1/2}dy_2$ and the same $\phi$ as before. Then,
the corresponding $(\phi,\beta)$-transformation  maps the previous ellipsoid into the one of axes $(1/d_1, 1/ d_2, 1, \ldots, 1)$.
Iteration of this procedure leads us towards the ellipsoid of axes $(1/d_1, \ldots, 1/d_n)$.
To show the second claim, observe that $\bar\alpha^2 = \alpha^2\,(1 - \sum_{i=1}^n s_i^2)$, where $s_i=\beta_i/\alpha$ and $\beta_i=(1-d_i^2)^{1/2}dy_i$.
\hfill$\Box$

\begin{proposition}\label{C-3.5}
If a Minkowski norm $F$ in $\RR^n$ can be $T_{\beta,\phi}$-deformed to the Euclidean norm $\alpha$
then
the Cartan torsion of $F$ is semi-C-reducible.
\end{proposition}

\proof This follows from \eqref{E-c-value0-2} and Theorem~\ref{prop:inverse-o}.
\hfill$\Box$

\smallskip

By Proposition~\ref{C-3.3} with $p=1$, any $(\alpha,\beta)$-norm can be deformed into a Euclidean norm
by one $T_{\beta,\phi}$-deformation.
Thus, by Proposition~\ref{C-3.5}, $(\alpha,\beta)$-norms are semi-C-reducible.

\begin{example}\rm
Given $\phi(s)$ and 1-form $\beta$ in $\RR^n$,
one can study the iterations of $T_{\beta,\phi}$-transformati\-on of the space
${\rm Mink}^n$
equipped with, say, the Hausdorff distance $d_H$.
These iterations define a dynamical system in this metric space $({\rm Mink}^n, d_H)$.
One could try to study its dynamics: fixed or periodic points, limit sets, etc.
 Let $F_1=F\phi(\beta/F)$, $F_2=F_1\phi(\beta/{F_1})$ and so on.
Then $F_2=F\psi_1(\beta/F)$, where $\psi_1(s)=\phi(s)\phi(\frac{s}{\phi(s)})$ and $s=\beta/F$,
see Proposition~\ref{prop:compos}, and so on. Notice~that
\[
 F\psi_{k+1}(\beta/F) =
 F_{k+1}= F\psi_k\Big(\beta/F)\phi\Big(\frac\beta{F\psi_k(\beta/F)}\Big).
\]
The functions $\psi_k(s)$ satisfy the following recurrence relation:
 $\psi_{k+1}(s)=\psi_k(s)\phi\big({s}/{\psi_k(s)}\big)$.
If there exists a positive function $\psi_\infty=\lim_{\,k\to\infty}\psi_k$
then it is unique and $\phi(\frac{s}{\psi_\infty(s)})=1$.
The~solu\-tion is $\psi_\infty=s/s_0$, where $\phi(s_0)=1$.
Then
 $F_\infty=F\psi_\infty(\beta/F) = \beta/{s_0}$
-- the indicatrix $F_k = 1$ converges to the hyperplane $\beta = s_0$.
The
$F_\infty$-``norm" is highly singular: all the $g_y$'s are identically~zero.

For Randers deformation, with $\phi(s) = 1+s$, we get  $F_k = F + k\beta$,
so $F_k$ stops to be ``true Minkowski norm" for $k$ larger than $1/F(\beta)$, and, therefore, further iterations make no sense.

For Kropina deformation, $\phi (s) = 1/s$, we get
$\psi_k(s) = s^{1 - 2^{k+1}}$, therefore, $\psi_\infty(s)  = 0$ for $s> 1$,
$\psi_\infty(1) = 1$, and $\psi_\infty (s) = + \infty$ for $0< s <1$, the function is not positive.
\end{example}



\end{document}